%% file: v1_arxiv.tex
\documentclass[12pt]{amsart}

\usepackage{amsfonts, amstext, amsmath, amsthm, amscd, amssymb, enumitem, url}
\usepackage[table]{xcolor}
\usepackage{tikz-cd}
\usepackage{pdfpages}
\usepackage{pdflscape}

\usepackage{geometry}
\usepackage[parfill]{parskip}
\usepackage{microtype}

\usepackage{tabularx}

\newcolumntype{x}{>{\centering\arraybackslash}X}
\usepackage{makecell}
\usepackage{longtable}
\usepackage{hhline}
\usepackage{multirow}
\usepackage{subfig}

\usetikzlibrary{arrows.meta}

\usepackage[colorlinks,citecolor=cyan,linkcolor=magenta]{hyperref}
\usepackage{zref-clever} 
\zcsetup{cap = true, nameinlink = true}
\newcommand{\Cref}[1]{\zcref{#1}}

\newtheorem*{prop:surtototallyperiodicrecog}{\Cref{prop:surtototallyperiodicrecog}}
\newtheorem*{prop:freepiecepartialsection}{\Cref{prop:freepiecepartialsection}}

\newtheorem{thm}{Theorem}[section]
\zcRefTypeSetup{thm}{Name-sg = Theorem}

\usepackage{xparse} 

\NewCommandCopy{\newtheoremcopy}{\newtheorem}
\RenewDocumentCommand{\newtheorem}{m O{thm} m}{ 
\newtheoremcopy{#1}[#2]{#3}
\AddToHook{env/#1/begin}{\zcsetup{countertype={thm=#1}}}
\zcRefTypeSetup{#1}{Name-sg = #3}}

\newtheorem{lem}[thm]{Lemma}
\newtheorem{prop}[thm]{Proposition}
\newtheorem{cor}[thm]{Corollary}
\newtheorem{claim}[thm]{Claim}
\newtheorem{conj}[thm]{Conjecture}

\RenewDocumentCommand{\newtheorem}{m O{thm} m}{ 
\newtheoremcopy{#1}[#2]{#3}
\AddToHook{env/#1/begin}{\zcsetup{countertype={thm=#1}}}
\zcRefTypeSetup{#1}{Name-sg = #3}
\AtBeginEnvironment{#1}{\pushQED{\hfill $\lozenge$}}
\AtEndEnvironment{#1}{\popQED}}

\theoremstyle{definition}
\newtheorem{defn}[thm]{Definition}

\newtheorem{eg}[thm]{Example}

\numberwithin{equation}{section}

\usepackage{todonotes}
\renewcommand{\epsilon}{\varepsilon}

\usepackage{stmaryrd}
\newcommand{\cut}{\!\bbslash\!}

\newcommand{\orb}{\mathrm{orb}}

\DeclareMathOperator{\ind}{\mathrm{ind}}

\DeclareMathOperator{\intr}{\mathrm{int}}

\allowdisplaybreaks

\begin{document}

\title[From graph manifolds to totally periodic flows]{From pseudo-Anosov flows on graph manifolds to totally periodic flows}

\author{Chi Cheuk Tsang}
\address{Département de mathématiques \\
Université du Québec à Montréal \\
201 President Kennedy Avenue \\
Montréal, QC, Canada H2X 3Y7}
\email{tsang.chi\_cheuk@uqam.ca}

\begin{abstract}
We show that every pseudo-Anosov flow on a graph manifold is almost equivalent, i.e. orbit equivalent in the complement of a finite collection of closed orbits, to a totally periodic pseudo-Anosov flow or a suspension Anosov flow.
The proof is via a hands-on construction of a partial Birkhoff section with genus one components that misses finitely many closed orbits.
When combined with previous work of the author, this implies that every transitive Anosov flow on a graph manifold with orientable stable and unstable foliations is almost equivalent to a suspension Anosov flow.
\end{abstract}

\maketitle


\section{Introduction} \label{sec:intro}

An \textbf{Anosov flow} on a closed orientable 3-manifold is a flow whose flow lines converge along leaves of a stable foliation and diverges along leaves of an unstable foliation.

As a simple example, the suspension flow of an Anosov map on the torus is an Anosov flow.
On the other hand, there are Anosov flows with much more complicated global dynamics.
For instance, there are
Anosov flows where every closed orbit is homotopic to infinitely many other closed orbits \cite{Fen94}, 
Anosov flows with infinitely many transverse tori \cite{BBY17}, and
Anosov flows whose stable and unstable foliations are non-homeomorphic \cite{BBY26}.

In view of this, one might find the following conjecture surprising.

\begin{conj}[{Fried \cite{Fri83}, Christy \cite[Problem 3.54]{Kir97}, Ghys}] \label{conj:friedchristyghys}
Any two transitive Anosov flows with orientable stable and unstable foliations are \textbf{almost equivalent}, in the sense that they differ only by reparametrization and homeomorphism in the complement of finitely many closed orbits. 
\end{conj}

See \cite[Section 1]{Tsa24b} for more discussion on the different formulations of this conjecture. 
We thank Bin Yu for pointing out the reference \cite[Problem 3.54]{Kir97}.

Nevertheless, there has been an increasing amount of evidence in support of \Cref{conj:friedchristyghys} in recent years. 
We now know that:
\begin{itemize}
    \item Any two suspension Anosov flows with orientable stable and unstable foliations are almost equivalent (\cite{DS19}).
    \item The geodesic flow of every hyperbolic orientable orbifold of `small type' is almost equivalent to a suspension Anosov flow (\cite{Fri83, Deh15, HM17, Deh26}).
    \item Every Anosov flow with orientable stable and unstable foliations admitting a Birkhoff section with Penner type first return map is almost equivalent to a suspension Anosov flow (\cite[Theorem 1.3]{Tsa24b}).
    \item Every totally periodic transitive Anosov flow with orientable stable and unstable foliations is almost equivalent to a suspension Anosov flow (\cite[Theorem 1.4]{Tsa24b}).
\end{itemize}
For the last item, a flow is \textbf{totally periodic} if it is defined on a \textbf{graph manifold}, i.e. a manifold where all JSJ pieces are Seifert fibered, and the fibers in each JSJ piece are homotopic to a closed orbit.
(For the terminology on the other items, we refer the reader to the corresponding references, since they will not play a role in this paper.)

The purpose of this article is to make an addition to these results. 
The main theorem is the following.

\begin{thm} \label{thm:graphmanifoldtototallyperiodic}
Every pseudo-Anosov flow on a graph manifold is almost equivalent to a totally periodic flow or a suspension Anosov flow.
\end{thm}

A \textbf{pseudo-Anosov flow} is a flow that is Anosov away from a finite collection of singular orbits.
We consider these flows more generally since it does not add any difficulty to the proof.
When combined with \cite[Theorem 1.4]{Tsa24b} as stated above, this gives the following corollary.

\begin{cor}
Every transitive Anosov flow with orientable stable and unstable foliations on a graph manifold is almost equivalent to a suspension Anosov flow.
\end{cor}

\subsection{Sketch of proof}

The proof of \Cref{thm:graphmanifoldtototallyperiodic} is explicit, and does not use the horizontal Goodman surgery machinery of \cite{Tsa24a} and \cite{Tsa24b}. 
Instead, it relies on two ingredients.
The first ingredient is a criterion for recognizing when a pseudo-Anosov flow is almost equivalent to a totally periodic flow or a suspension Anosov flow.

\begin{prop:surtototallyperiodicrecog}
Let $\phi$ be a pseudo-Anosov flow. Suppose $\phi$ admits a partial Birkhoff section $S$ satisfying:
\begin{itemize}
    \item every component of $S$ has genus one,
    \item every boundary component of $S$ has index $-1$, and
    \item the set of flow lines that are disjoint from $S$ is a finite collection of closed orbits.
\end{itemize}
Then $\phi$ is almost equivalent to totally periodic flow or a suspension Anosov flow.
\end{prop:surtototallyperiodicrecog}

Here, a \textbf{partial Birkhoff section} is an immersed, possibly disconnected, surface with boundary $S$ in $M$ where the interior $\intr(S)$ is embedded and transverse to $\phi$, and the boundary $\partial S$ lies along orbits of $\phi$.

A partial Birkhoff section can be transformed into a \textbf{partial cross section}, i.e. a partial Birkhoff section with empty boundary, by Goodman-Fried surgery. 
Applying this to the hypothesis of \Cref{prop:surtototallyperiodicrecog}, we get a disjoint union of tori that intersect all but finitely many closed orbits of the surgered flow. 
From this, one can deduce that the surgered flow is totally periodic.

The second ingredient is a construction of a partial Birkhoff section as in \Cref{prop:surtototallyperiodicrecog}.
To outline this, we have to first briefly explain the work of Barbot-Fenley \cite{BF13, BF15, BF21} on pseudo-Anosov flows on Seifert fibered JSJ pieces.
A Seifert fibered piece is \textbf{periodic} if its fiber is homotopic to a closed orbit. Otherwise it is \textbf{free}.
Periodic pieces contain only finitely many closed orbits, and JSJ tori between periodic pieces can be made transverse, capturing orbits that pass between the pieces.
Thus it remains to capture the orbits inside free pieces, or at least all but finitely many of them.

To this end, we use the fact that each free piece is a blowup of a finite fiberwise cover of the geodesic flow of a hyperbolic orbifold with boundary.
We construct partial Birkhoff sections for the latter.

\begin{prop:freepiecepartialsection}
Let $\Sigma$ be a hyperbolic orbifold, possibly non-orientable and possibly with boundary. 
Let $\phi$ be a $k$-fold fiberwise cover of the geodesic flow of $\Sigma$. 
Then there exists a partial Birkhoff section $S$ satisfying:
\begin{itemize}
    \item every component of $S$ has genus one,
    \item every boundary component of $S$ has index $-1$, and
    \item $S$ intersects every flow line except for a finite collection of closed orbits.
\end{itemize}
\end{prop:freepiecepartialsection}

The construction of these partial Birkhoff sections is inspired by the construction behind \cite{Deh26}, and is very hands-on. See the figures in \Cref{subsec:basicblock}.

Taking the union of the lifts of these partial Birkhoff sections, along with the transverse tori between periodic pieces, and applying \Cref{prop:surtototallyperiodicrecog} gives us \Cref{thm:graphmanifoldtototallyperiodic}.

\subsection*{Acknowledgement}

We thank Thomas Barthelmé, Sergio Fenley, and Federico Salmoiraghi for helpful discussions. 
We thank Pierre Dehornoy, Clément Perault, and Bin Yu for their comments on a draft of this paper.
This project was completed while the author is a CRM postdoctoral fellow based at CIRGET. We thank the center for its support.

\section{Background} \label{sec:background}

In this section, we recall some background on pseudo-Anosov flows and their interaction with the JSJ decomposition of 3-manifolds.

\subsection{Pseudo-Anosov flows}

A \textbf{flow} on a 3-manifold $M$ is a continuous map $\phi: \mathbb{R} \times M \to M$ satisfying $\phi(0,x) = x$ and $\phi(s,\phi(t,x)) = \phi(s+t,x)$ for every $s,t \in \mathbb{R}, x \in M$.
Two flows $\phi_1$ and $\phi_2$ are \textbf{orbit equivalent} if there is a homeomorphism sending the oriented flow lines of $\phi_1$ to that of $\phi_2$, while not necessarily preserving their parametrizations. 
Two flows $\phi_1$ and $\phi_2$ are \textbf{almost equivalent} if there exists a finite collection $\mathcal{C}_1$ of closed orbits of $\phi_1$ and a finite collection $\mathcal{C}_2$ of closed orbits of $\phi_2$, such that the restriction of $\phi_1$ to the complement of $\mathcal{C}_1$ is orbit equivalent to the restriction of $\phi_2$ to the complement of $\mathcal{C}_2$.

A \textbf{pseudo-Anosov flow} is a flow $\phi$ on a closed oriented 3-manifold $M$ for which there exists a pair of transverse singular 2-dimensional foliations $\mathcal{F}^s$ and $\mathcal{F}^u$, which we refer to as the \textbf{stable} and \textbf{unstable} foliations respectively, such that
\begin{itemize}
    \item $\mathcal{F}^s$-leaves and $\mathcal{F}^u$-leaves intersect in flow lines,
    \item flow lines along each $\mathcal{F}^s$-leaf converge in forward time and diverge in backward time, and
    \item flow lines along each $\mathcal{F}^u$-leaf diverge in forward time and converge in backward time.
\end{itemize}
For more details on the definition of pseudo-Anosov flows, we refer the reader to \cite[Chapter 1.1]{BM25}.

See \Cref{fig:paflow} left for a local picture of a pseudo-Anosov flow away from singular orbits, and \Cref{fig:paflow} right for a local picture near a 3-pronged singular orbit.

\begin{figure}
    \centering
    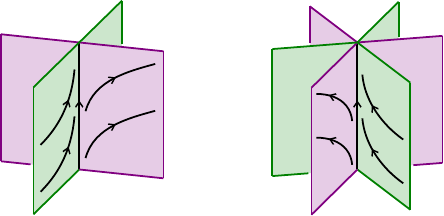
    \caption{Local picture of a pseudo-Anosov flow away from singular orbits (left) and near a 3-pronged singular orbit (right).}
    \label{fig:paflow}
\end{figure}

\begin{eg}[Suspension Anosov flow]
Let $f:T \to T$ be an Anosov map on a torus. Then the suspension flow on the mapping torus $M = T \times [0,1]/(x,1) \sim (f(x),0)$ is a pseudo-Anosov flow with no singular orbits. We refer to such a flow as a \textbf{suspension Anosov flow}.
\end{eg}

Suppose $\gamma$ is a closed orbit of a pseudo-Anosov flow $\phi$ on $M$.
Let $\nu$ be a small tubular neighborhood of $\gamma$.
The local stable and unstable leaves of $\gamma$ cut $\nu$ into multiple components. We refer to these components as the \textbf{quadrants} at $\gamma$.
Also, the intersection of the local stable leaves of $\gamma$ with $\partial \nu$ determines a multicurve on $\partial \nu$. We call the isotopy class of this multicurve the \textbf{degeneracy curve} at $\gamma$.

Suppose $s$ is a slope on $\partial \nu$. Let $p$ be geometric intersection number between $s$ and $d$. 
If $p \geq 2$, then there exists a pseudo-Anosov flow $\phi_s(\gamma)$ on the Dehn surgered manifold $M_s(\gamma)$ obtained by `blowing up' $\phi$ at $\gamma$ then `blowing down' along $s$. 
See \Cref{fig:gfsurgery}, ignoring the gray surface for now.
In particular, $\phi_s(\gamma)$ has a $p$-pronged closed orbit in the core of the filling solid torus, which we refer to as the \textbf{filled orbit}, and the restriction of $\phi_s(\gamma)$ to the complement of the filled orbit is orbit equivalent to the restriction of $\phi$ to the complement of $\gamma$.
The latter property implies that $\phi$ and $\phi_s(\gamma)$ are almost equivalent.
We will refer to $\phi_s(\gamma)$ as a \textbf{Goodman-Fried surgery} of $\phi$ along $\gamma$.
See \cite{Fri83} for details.

\begin{figure}
    \centering
    \fontsize{8pt}{8pt}
    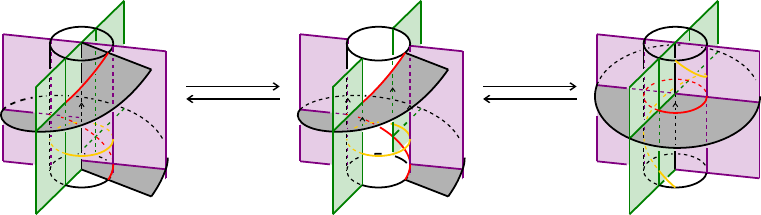
    \caption{Goodman-Fried surgery is performed by blowing up along a closed orbit $\gamma$ and blowing down along a different slope. It can be used to transform partial Birkhoff sections into partial cross sections.}
    \label{fig:gfsurgery}
\end{figure}

Finally, we will need the following version of the closing lemma in \Cref{subsec:bookofIbundles}.

\begin{prop}[{\cite[Proposition 1.4.4]{BM25}}] \label{prop:closing}
Let $K$ be a compact subset in the complement of the singular orbits. For every $\epsilon > 0$, there exists $\delta > 0$ such that for every $x \in K$, if $d(x,\phi(T,x))<\delta$ for large enough $T$, then there exists a closed orbit in the $\epsilon$-neighborhood of the orbit segment from $x$ to $\phi(T,x)$. \qed
\end{prop}

\subsection{Partial Birkhoff sections}

Let $\phi$ be a pseudo-Anosov flow on $M$.
A \textbf{partial Birkhoff section} to $\phi$ is an immersed, possibly disconnected, surface with boundary $S$ in $M$ where the interior $\intr(S)$ is embedded and transverse to $\phi$, and the boundary $\partial S$ lies along orbits of $\phi$.
A \textbf{partial cross section} to $\phi$ is a partial Birkhoff section that has empty boundary.

Suppose $c$ is a boundary component of a partial Birkhoff section $S$, lying along a closed orbit $\gamma$ of $\phi$.
Let $\nu$ be a small tubular neighborhood of $\gamma$. Then $S$ induces a slope $s$ on $\partial \nu$. 
Let $p$ be the geometric intersection number between $s$ and the degeneracy slope $d$ at $\gamma$. 
We define the \textbf{index} of $c$ to be the non-positive half-integer $\ind(c) = -\frac{p}{2}$.
Note that since $S$ is embedded, any other boundary component of $S$ lying along $\gamma$ will automatically have the same index.

Meanwhile, since the interior $\intr(S)$ is transverse to $\phi$, the stable foliation of $\phi$ induces a singular foliation on $\intr(S)$. We define the \textbf{index} of a $q$-pronged singular point $x$ to be the negative half-integer $\ind(x) = 1 - \frac{q}{2}$. 
Under this notation, the Poincaré-Hopf theorem states that
\begin{equation} \label{eq:indPH}
\chi(S) = \sum_c \ind(c) + \sum_x \ind(x)
\end{equation}
where the first sum is over boundary components of $S$ and the second sum is over singular points in $\intr(S)$.

Partial Birkhoff sections can be transformed into partial cross sections via Goodman-Fried surgery.
In more detail, suppose $c$ is a boundary component of a partial Birkhoff section $S$ with index $\ind(c) \leq -1$.
Then the geometric intersection number between the slope $s$ of $S$ and the degeneracy slope at $c$ is $p \geq 2$, thus one has the Goodman-Fried surgered flow $\phi_s(\gamma)$ on $M_s(\gamma)$. 
Up to an isotopy of $S$ along flow lines of $\phi$, the restriction of $S$ to the complement of $\gamma$ can be completed into a surface with boundary $S_s(\gamma)$ in $M_s(\gamma)$ by adding finitely many points on the filled orbit of $\phi_s(\gamma)$.
See \Cref{fig:gfsurgery}.
Then $S_s(\gamma)$ is a partial Birkhoff section to $\phi_s(\gamma)$, and it can be topologically obtained from $S$ by collapsing each boundary component of $S$ lying along $\gamma$ down to a point.

Performing this procedure at each boundary component of $S$, we obtain a partial cross section $R$ that can be topologically obtained from $S$ by collapsing each boundary component down to a point.
In particular, the components of $R$ can be canonically identified with those of $S$, and each corresponding component has the same genus.

\subsection{Quasi-transverse tori}

A particular kind of Birkhoff partial section plays a central role in the study of pseudo-Anosov flows:
A \textbf{Birkhoff annulus} is a Birkhoff partial section that is homeomorphic to an annulus.
From \Cref{eq:indPH}, we deduce that each Birkhoff annulus must be disjoint from the singular orbits, and both of its boundary components must have index $0$.
See \Cref{fig:maximallytransverse} left for a local picture of two Birkhoff annuli.

\begin{figure}
    \centering
    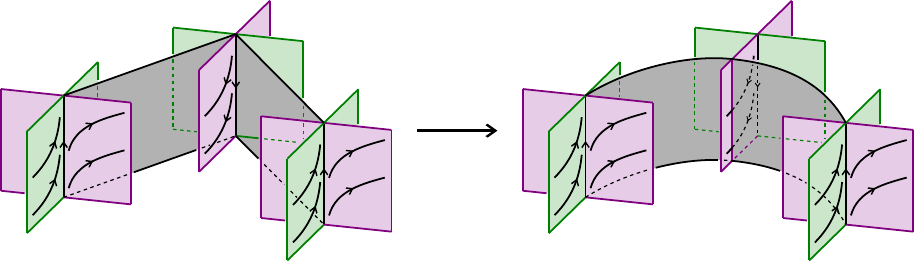
    \caption{If two adjacent Birkhoff annuli in a quasi-transverse torus $T$ occupy adjacent quadrants at a boundary orbit, then $T$ can be homotoped to remove one tangent orbit.}
    \label{fig:maximallytransverse}
\end{figure}

An immersed torus $T$ is \textbf{quasi-transverse} if it is a union of Birkhoff annuli $A_1,...,A_n$.
Furthermore, $T$ is \textbf{weakly embedded} if $A_i$ have disjoint interiors.
All quasi-transverse tori will be weakly embedded in this paper.

Suppose $A_1$ and $A_2$ are adjacent Birkhoff annuli in a quasi-transverse torus $T$. 
Let $\gamma$ be the orbit that their common boundary component lie along.
If $A_1$ and $A_2$ occupy adjacent quadrants at $\gamma$, then $T$ can be homotoped by pushing it along a stable or unstable half-leaf of $\gamma$ in order to remove one tangent orbit. 
See \Cref{fig:maximallytransverse}.

If no two adjacent Birkhoff annuli in a quasi-transverse torus $T$ occupy adjacent quadrants, then we say that $T$ is \textbf{maximally transverse}.
All quasi-transverse tori will be maximally transverse in this paper.

\begin{prop} \label{prop:quasitransversetori}
Let $\phi$ be a pseudo-Anosov flow on $M$.
Every $\pi_1$-injective embedded torus or Klein bottle $T$ in $M$ can be homotoped into a transverse or maximally transverse weakly embedded quasi-transverse torus $T'$. 
Moreover, $T'$ is unique up to homotopy along flow lines. 
\end{prop}
\begin{proof}
The first statement, without the specification on `maximally transverse', is proved in \cite[Theorem 6.10]{BF13}. See also \cite[Theorem 5.3.10]{BM25}.
The fact that one can arrange for maximally transverse is recorded in \cite[Remark 5.3.15]{BM25}.
In order to explain the second statement, let us briefly recap the proof.

The \textbf{orbit space} of $\phi$ is the quotient of the universal cover $\widetilde{M}$ by the lifted flow lines.
This is a space $\mathcal{O}$ homeomorphic to a plane with a pair of singular 1-dimensional foliations $\mathcal{O}^s$ and $\mathcal{O}^u$, see \cite[Chapter 1.3]{BM25}.
A \textbf{lozenge} in $\mathcal{O}$ is the interior of a properly embedded, trivially foliated, copy of $[0,1]^2 \backslash \{(0,0),(1,1)\}$.

One can show that either $\phi$ is a suspension Anosov flow, in which case $T$ can be homotoped into a fiber of the underlying Anosov mapping torus, uniquely up to homotopy along flow lines, or the action of $\pi_1 T$ on $\mathcal{O}$ preserves a unique linear chain $\mathcal{C}$ of lozenges.
For each lozenge $L$ in $\mathcal{C}$, one can construct a Birkhoff annulus $A$ with a lift $\widetilde{A}$ that projects down to $L$ and its two (non-ideal) corners.
Taking the union of such Birkhoff annuli over the $\pi_1 T$-orbits of lozenges in $\mathcal{C}$, one gets a quasi-transverse torus $T'$ homotopic to $T$ and with a lift $\widetilde{T'}$ that projects down to $\mathcal{C}$ and the corners of the lozenges. 
One can show that $T'$ is weakly embedded.

If two adjacent lozenges share a side, then they must occupy adjacent quadrants at a corner $x$. In turn, their corresponding Birkhoff annuli occupy adjacent quadrants at a boundary orbit, thus one can homotope $T'$ to remove one tangent orbit as described above.
After homotopy, the corner $x$ (and its $\pi_1 T$-translates) are removed from the projection of $\widetilde{T'}$. 
Repeating this procedure, one arranges for $T'$ to be maximally transverse, at which point the projection of $\widetilde{T'}$ is $\mathcal{C}$ and the corners of lozenges that do not lie along a shared side.

Conversely, if $T''$ is a maximally transverse quasi-transverse torus homotopic to $T$, then it will have a lift $\widetilde{T''}$ that projects down to $\mathcal{C}$ and the corners of the lozenges that do not lie along a shared side. See \cite[Proposition 5.3.5]{BM25}.
Thus $\widetilde{T'}$ and $\widetilde{T''}$ are homotopic along flow lines. 
Projecting the homotopy down to $M$, we see that $T'$ and $T''$ are homotopic along flow lines. This proves the second statement.
\end{proof}

\subsection{JSJ decomposition}

Let $M$ be an irreducible orientable closed 3-manifold. 
A \textbf{torus decomposition} of $M$ is a collection of embedded tori $\{T_1,...,T_n\}$ so that the 3-manifold $M \cut (\bigcup_{i=1}^n T_i)$ obtained by cutting $M$ along $\bigcup_{i=1}^n T_i$ is a disjoint union of torus (semi-)bundles, Seifert fibered manifolds, and atoroidal manifolds.
The \textbf{JSJ decomposition} of $M$ is the minimal torus decomposition of $M$. It is unique up to isotopy. 
See \cite{JS79, Joh79, NS97} for proofs.
We refer to the elements $T_i$ of the JSJ decomposition as the \textbf{JSJ tori}, and refer to the components of the cut 3-manifold $M \cut \bigcup_{i=1}^n T_i$ as the \textbf{JSJ pieces}, or just \textbf{pieces} for short.

We say that $M$ is a \textbf{graph manifold} if all of its JSJ pieces are Seifert fibered manifolds.
Note that our convention is that Seifert fibered manifolds are graph manifolds, while torus (semi-)bundles are \emph{not} graph manifolds. 

Now let $\phi$ be a pseudo-Anosov flow on $M$.
By \Cref{prop:quasitransversetori}, the JSJ tori can be homotoped into quasi-transverse tori. 
This allows us to focus on the restriction of $\phi$ to each JSJ piece of $M$.

We say that a Seifert fibered piece $P$ of $M$ is \textbf{periodic} if under some choice of Seifert fibering, the fibers are homotopic to a closed orbit of $\phi$.
Otherwise, we say that $P$ is \textbf{free}.
We say that $\phi$ is \textbf{totally periodic} if all JSJ pieces of $M$ are periodic Seifert fibered pieces.

Periodic Seifert fibered pieces admit the following description.

\begin{prop}[{\cite[Theorem F]{BF13}}] \label{prop:periodicpiece}
Let $P$ be a periodic Seifert fibered piece of $M$. Then there exists a union of Birkhoff annuli with disjoint interiors $Z$, so that a neighborhood of $Z$ is isotopic to $P$.
\end{prop}

We also need the following fact about JSJ tori between periodic Seifert fibered pieces.

\begin{prop}[{\cite[Theorem B]{BF15}}] \label{prop:toribetweenperiodicpieces}
Let $T$ be a JSJ torus between two periodic Seifert fibered pieces of $M$. Then $T$ is isotopic to a partial cross section.
\end{prop}

Before stating the description for free Seifert fibered pieces, we set up some terminology.

\begin{defn}[Flows on 3-manifold with boundary] \label{defn:flowmanifoldwithboundary}
A \textbf{flow} on a 3-manifold with boundary $P$ is a continuous map $\phi: K \to P$ defined on a compact subset $K \subset \mathbb{R} \times P$ that contains $\{0\} \times P$, satisfying $\phi(0,x) = x$ and $\phi(s,\phi(t,x)) = \phi(s+t,x)$ for every $s,t \in \mathbb{R}, x \in P$ for which both sides of the equations are defined.

We define the \textbf{incoming boundary} $\partial_- P$ to be the set of points $x \in \partial P$ where $\phi(t,x)$ is not defined for $t < 0$, and the \textbf{outgoing boundary} $\partial_+ P$ to be the set of points on $x \in \partial P$ where $\phi(t,x)$ is not defined for $t > 0$.

For all flows in this paper, $\partial_- P$ and $\partial_+ P$ will be subsurfaces with boundary of $\partial P$, and the \textbf{tangential boundary} $\partial_0 P$, which we define to be the complement of the interiors of $\partial_- P$ and $\partial_+ P$, will either be a union of curves or a union of annuli.
\end{defn}

\begin{eg}[Geodesic flow]
Let $\Sigma$ be a hyperbolic orbifold, possibly non-orientable and possibly with boundary. The \textbf{unit tangent bundle} of $\Sigma$ is the Seifert fibered space $T^1 \Sigma = \{v \in T \Sigma \mid ||v|| = 1\}$ with base orbifold $\Sigma$. The \textbf{geodesic flow} on $T^1 \Sigma$ is defined by $\phi(t,\dot{\alpha}(0)) = \dot{\alpha}(t)$ for every unit speed geodesic $\alpha$.
In this case, both the incoming and outgoing boundary consists of one annulus, while the tangential boundary consists of two closed orbits per boundary torus of $T^1 \Sigma$.
See \Cref{fig:freepiece} left for a local picture near a boundary torus.
\end{eg}

\begin{figure}
    \centering
    \fontsize{10pt}{10pt}
    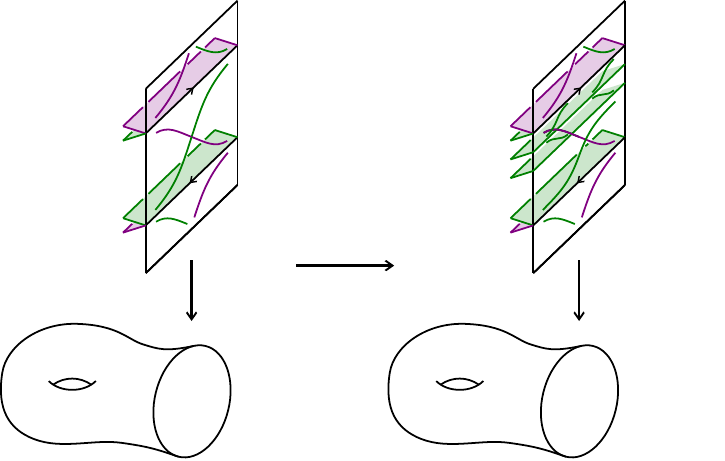
    \caption{Local picture of a hyperbolic blow up of the geodesic flow.}
    \label{fig:freepiece}
\end{figure}

Two flows $\phi_1$ and $\phi_2$ on are \textbf{orbit equivalent} if there is a homeomorphism sending the oriented flow lines of $\phi_1$ to that of $\phi_2$, while not necessarily preserving their parametrizations.
The flow $\phi_1$ is \textbf{orbit semi-equivalent} to $\phi_2$ if there is a continuous map $f$ sending the oriented flow lines of $\phi_1$ locally homeomorphically to that of $\phi_2$, while not necessarily preserving their parametrizations, and such that every flow line of $\phi_2$ intersects the image of $f$.
This terminology is based on the notion of conjugacy and semi-conjugacy in dynamical systems.

\begin{prop}[{\cite[Theorem 7.1]{BF21}}] \label{prop:freepiece}
Let $P$ be a free Seifert fibered piece of $M$.
Then up to homotoping $P$, the restriction of $\phi$ to $P$ is orbit semi-equivalent to a finite fiberwise cover of the geodesic flow of a (not necessarily orientable) hyperbolic orbifold $\Sigma$ with geodesic boundary. 
\end{prop}
\begin{proof}
In \cite{BF21}, it is shown that up to homotoping $P$, the restriction of $\phi$ to $P$ is orbit equivalent to a `hyperbolic blow up' along boundary closed orbits of a finite fiberwise cover of the geodesic flow of a hyperbolic orbifold with geodesic boundary $\Sigma$.
See \Cref{fig:freepiece} for a local picture.
The exact definition for hyperbolic blow up is unimportant for this paper.
The only fact we will need is that there is an orbit semi-equivalence from the hyperbolic blow up to the original flow. This is the map $\chi_\Gamma$ constructed in \cite[Section 4.3.3]{BF21}.

Also, the statement in \cite{BF21} has a hypothesis that $P$ is \textbf{non-elementary}, i.e. $\pi_1 P$ does not contain a free abelian group of finite index. 
This is used to rule out the case when the entire flow $\phi$ is a suspension Anosov flow:
If $M$ is an Anosov mapping torus, the convention of JSJ decomposition in \cite{BF21} is that one has to cut $M$ along a torus fiber, resulting in a free Seifert fibered piece.
However, in our convention, the JSJ decomposition of $M$ is empty.

In fact, for us, the only elementary Seifert fibered JSJ pieces are the closed flat 3-manifolds and the twisted $I$-bundle over the Klein bottle. 
But the former cannot admit pseudo-Anosov flows by \cite[Corollary 3.2.2]{BM25} and the latter must be free by \Cref{prop:quasitransversetori}, so we can omit the non-elementary hypothesis.
\end{proof}

\section{Recognition of totally periodic flows}

The main goal of this section is to establish a criterion for a pseudo-Anosov flow being almost equivalent to a totally periodic flow or a suspension Anosov flow (\Cref{prop:surtototallyperiodicrecog}).
To that end, we first prove a criterion for recognizing totally periodic flows in \Cref{subsec:bookofIbundles}.
Actually, we will prove a more general criterion for recognizing books of I-bundles (\Cref{prop:bookofIbundlesrecog}), in anticipation of future applications.

\subsection{Book of I-bundles} \label{subsec:bookofIbundles}

Recall the terminology for flows on 3-manifolds with boundary from \Cref{defn:flowmanifoldwithboundary}.

\begin{eg}[Product]
Let $S$ be a surface, possibly with boundary. We refer to $P = S \times [0,1]$ along with the flow $\phi(s,(x,t)) = (x,s+t)$ as a \textbf{product}. 
Under this flow, the incoming boundary $\partial_- P = S \times \{0\}$, the outgoing boundary $\partial_+ P = S \times \{1\}$, and the tangential boundary $\partial_0 P = \partial S \times [0,1]$.
\end{eg}

\begin{eg}[Round handles]
Consider the vector field $\widetilde{X} = x \frac{\partial}{\partial x} - y \frac{\partial}{\partial y} + \frac{\partial}{\partial z}$ on the set $\widetilde{P} = \{(x,y,z) \in \mathbb{R}^3 \mid |x| \leq 1, |y| \leq 1, |xy| \leq \frac{1}{2} \}$.
It descends to a vector field $X$ on the quotient $P = \widetilde{P} / (x,y,z) \sim (x,y,z+1)$.
Let $\phi$ be the flow on $P$ generated by $X$.
Under this flow, the incoming boundary $\partial_- P$ consists of two annuli, the outgoing boundary $\partial_+ P$ consists of two annuli, and the tangential boundary $\partial_0 P$ consists of four annuli.
See \Cref{fig:roundhandle}.
We refer to $P$ with the flow $\phi$ as a \textbf{round handle}.

\begin{figure}
    \centering
    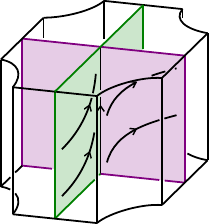
    \caption{The prototypical round handle.}
    \label{fig:roundhandle}
\end{figure}

More generally, we refer to any finite branched cover of $P/(x,y,z) \sim (-x,-y,z)$ which branches over the core orbit $\{(0,0,z) \mid z \in \mathbb{R}/\mathbb{Z}\}$ with degree at least $2$, along with the lifted flow, as a \textbf{round handle} as well.
\end{eg}

\begin{defn}[Book of I-bundles]
A \textbf{book of I-bundles} is a 3-manifold with boundary obtained by homeomorphically gluing up the tangential boundaries of a disjoint union of products and round handles, in a way so that flow segments are glued to flow segments, along with the glued up flow.
In particular, every tangential boundary component has to be glued up, so a book of I-bundles has no tangential boundary.

Note that we allow for the possibility of no products or no round handles in the construction of a book of I-bundles.
In particular, a product with no tangential boundary is a book of I-bundles.
\end{defn}

\begin{prop} \label{prop:bookofIbundlesrecog}
Let $\phi$ be a pseudo-Anosov flow on $M$. Suppose $\phi$ admits a partial cross section $S$, and suppose that set of flow lines that are disjoint from $S$ is a finite collection of closed orbits. Then $\phi$ admits a partial cross section $S' \supset S$ such that the restriction of $\phi$ to $M \cut S'$ is a disjoint union of books of I-bundles.

Moreover, if every component of $S$ is a torus, then every component of $S'$ can be arranged to be a torus as well.
\end{prop}
\begin{proof}
Consider the restriction of $\phi$ to the cut manifold $M \cut S$. 
If a component $P$ of $M \cut S$ contains no closed orbits, then we claim that it is a product.

To see this, note that the flow line starting at each point $x$ on the incoming boundary $\partial_- P$ exits through the outgoing boundary $\partial_+ P$ at some finite time $t(x)$.
Indeed, otherwise the flow line has an accumulation point.
If the accumulation point lies on a singular orbit, then the whole singular orbit lies in the accumulation set, thus in $P$.
If the accumulation point lies away from the singular orbits, then one can construct a closed orbit contained in $P$ by applying \Cref{prop:closing} near the accumulation point.
But then the flow $\phi$ must map $\{(x,t) \in \partial_- P \times \mathbb{R} \mid t \in [0,t(x)]\}$ homeomorphically to $P$, certifying that it is a product.

If $P$ contains a closed orbit $\gamma$, we first claim that each flow line on the component of the stable leaf of $\gamma$ restricted to $P$ must enter from $\partial_- P$ at some finite time.
Suppose otherwise, then some such flow line $\ell$ never exits $P$ in backward time. But since $\ell$ lies on the stable leaf of $\gamma$, it spirals into $\gamma$ thus is non-closed, and it never exits $P$ in forward time as well.
This contradicts our hypothesis that $S$ intersects every non-closed flow line.

This implies that each stable half-leaf of $\gamma$ is an annulus with one boundary component on $\gamma$ and another boundary component being a circle on $\partial_- P$.
Let $N$ be a thin neighborhood of the union of stable half-leaves of $\gamma$. The boundary of $N$ is the union of two collections of annuli, one collection $A_-$ lying on $\partial_- P$ while the other collection $A_P$ lies within $P$.
Let $R = (S \backslash A_-) \cup A_P$. See \Cref{fig:bookofIbundlesrecog}.
Up to a small isotopy, we can make $R$ transverse to $\phi$.

\begin{figure}
    \centering
    \fontsize{10pt}{10pt}
    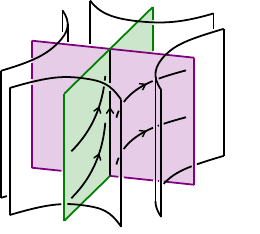
    \caption{In each component $P$ of $M \cut S$ that contains a closed orbit, we can construct a new surface $R$ that cuts $P$ into a book of I-bundles and a piece containing one less closed orbit.}
    \label{fig:bookofIbundlesrecog}
\end{figure}

Cutting along $R$ disconnects $P$. On one side we have a book of I-bundles, obtained by gluing up the product $(\partial_- P \cut c^s) \times I$ and the round handle that is a neighborhood of $\gamma$. On the other side we have a 3-manifold with boundary containing one less closed orbit of $\phi$.
Thus repeating this process, we can decompose $P$ into a disjoint union of book of I-bundles.

It remains to justify the last statement of the proposition. If every component of $S$ is a torus, then in the construction above, $\partial_- P$ is a union of tori. The circles in $c^s$ must each be essential, since otherwise $\gamma$ is a nullhomotopic closed orbit of $\phi$. Thus $\partial_- P \cut c^s$ is a disjoint union of annuli and tori, and $R$ is a union of tori.
\end{proof}

\subsection{Surgery to totally periodic flows}

We are now ready to prove our criterion for being almost equivalent to a totally periodic flow or a suspension Anosov flow.

\begin{prop} \label{prop:surtototallyperiodicrecog}
Let $\phi$ be a pseudo-Anosov flow. Suppose $\phi$ admits a partial Birkhoff section $S$ satisfying:
\begin{itemize}
    \item every component of $S$ has genus one,
    \item every boundary component of $S$ has index $-1$, and
    \item the set of flow lines that are disjoint from $S$ is a finite collection of closed orbits.
\end{itemize}
Then $\phi$ is almost equivalent to totally periodic flow or a suspension Anosov flow.
\end{prop}
\begin{proof}
Since every boundary component of $S$ has index $-1$, we can perform Goodman-Fried surgery to make $S$ into a partial cross section. 
After surgery, it remains true that every component of $S$ has genus one. 
Thus we can apply \Cref{prop:bookofIbundlesrecog} to extend $S$ so that the restriction of $\phi$ to $M \cut S$ is a disjoint union of books of I-bundles.
If every component of $M \cut S$ is a product, then $\phi$ is a suspension Anosov flow.
Otherwise, up to discarding parallel components, we can assume that no component of $M \cut S$ is a product.

Since each component of $S$ is a torus, each component of $M \cut S$ is Seifert fibered. Thus the JSJ pieces of $M$ are all Seifert fibered.
(But note that we do not claim that $M \cut S$ are exactly the JSJ pieces of $M$.)
Moreover, since the core orbit of each round handle in $M \cut S$ is a fiber, each JSJ piece of $M$ is periodic, thus the flow is totally periodic.
\end{proof}

\section{Partial sections in free Seifert pieces}

In this section, we will construct partial Birkhoff sections in finite fiberwise covers of geodesic flows.
Here, we extend the notion of partial Birkhoff sections to flows on manifolds with boundary with the same definition:
A \textbf{partial Birkhoff section} to a flow $\phi$ is an immersed, possibly disconnected, surface with boundary where the interior is embedded and transverse to $\phi$, and the boundary lies along orbits of $\phi$.

The objective is to show the following proposition.

\begin{prop} \label{prop:freepiecepartialsection}
Let $\Sigma$ be a hyperbolic orbifold, possibly non-orientable and possibly with boundary.
Let $\phi$ be a $k$-fold fiberwise cover of the geodesic flow of $\Sigma$. 
Then there exists a partial Birkhoff section $S$ satisfying:
\begin{itemize}
    \item every component of $S$ has genus one,
    \item every boundary component of $S$ has index $-1$, and
    \item $S$ intersects every flow line except for a finite collection of closed orbits.
\end{itemize}
\end{prop}

The strategy of the proof is as follows:
We first explicitly construct a partial Birkhoff section $S$ in the case where $\Sigma$ is a pair of pants and $k=1$.
This $S$ is homeomorphic to three once-punctured tori, and intersects every flow line except for the closed orbits lying over the boundary components of $\Sigma$. 

In fact, the same construction works even after capping off the boundary components of $\Sigma$ by cone points, i.e. if $\Sigma$ is an annulus with one cone point, or a disc with two cone points, or a sphere with three cone points.
We refer to this list of orbifolds as the \textbf{basic orbifolds}.

Since $S$ is a once-punctured torus, we can take its lift to get a genus one partial Birkhoff section even for other values of $k$.
The lifted partial Birkhoff section still has the same property of intersecting every flow line except for the boundary closed orbits.
In particular, this proves \Cref{prop:freepiecepartialsection} in the cases where $\Sigma$ is a basic orbifold.
This result will be recorded as \Cref{lem:basicblock}.

To prove \Cref{prop:freepiecepartialsection} for general $\Sigma$, we decompose $\Sigma$ into basic orbifolds, then place a partial Birkhoff section above each basic orbifold as in \Cref{lem:basicblock}.
The union of these partial Birkhoff sections will intersect every flow line except for the closed orbits lying above the decomposition curves and arcs.

\subsection{Basic orbifolds} \label{subsec:basicblock}

Let $\Sigma$ be a hyperbolic pair of pants. 
We label the boundary components of $\Sigma$ as $c_0, c_1, c_2$.
Let $\alpha$ be the oriented geodesic on $\Sigma$ depicted in \Cref{fig:freepiecesectioncurve}.
Note that it has three self-intersection points, each lying on the orthogeodesic between two boundary components.
It has five complementary regions, two being triangles and the remaining three each containing one boundary component.
We denote the latter three complementary regions as $\Sigma_0$, $\Sigma_1$, $\Sigma_2$ according to which boundary component they contain.
Let $\gamma$ be oriented lift of $\alpha$.
We will show that $\gamma$ bounds three partial Birkhoff sections $S_1, S_2, S_3$, each being a once-punctured torus.

\begin{figure}
    \centering
    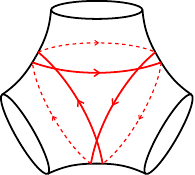
    \caption{The oriented geodesic $\alpha$ used to construct a basic section.}
    \label{fig:freepiecesectioncurve}
\end{figure}

To construct $S_1$, we topologically identify $\Sigma$ with the surface $(\mathbb{R}/\mathbb{Z} \times [0,2]) \backslash B_{\frac{1}{4}}(0,1)$, with $c_0 = \mathbb{R}/\mathbb{Z} \times \{0\}$, $c_1 = \partial B_{\frac{1}{4}}(0,1)$, and $c_2 = \mathbb{R}/\mathbb{Z} \times \{2\}$.
We identify the geodesic $\alpha$ with the red oriented curve shown in \Cref{fig:freepiecesectionbasic} top.
We also identify the self-intersection point of $\alpha$ lying on the orthogeodesic between $c_0$ and $c_2$ with the point $x = (\frac{1}{2},1)$.

We then pick two families of convex oriented loops based at $x$, illustrated as the yellow loops in \Cref{fig:freepiecesectionbasic} top.
The (unitized) tangent vectors of each family of loops form a hexagonal surface that projects homeomorphically onto a triangular complementary region of $\alpha$ on $\Sigma$. 
Three of the six sides of each hexagon lie along the orbit $\gamma$. 
These sides come from segments of the loops limiting onto $\alpha$ as oriented geodesic segments.

Each of the remaining three sides are arcs that lie above each of the three self-intersection points of $\alpha$.
These come from points on the loop limiting towards the self-intersection points.
Importantly, observe that these arcs coincide for the two hexagons.
We depict these arcs in \Cref{fig:freepiecesectionbasic} top as an interval of unit tangent vectors at each self-intersection point.
We also illustrate a local picture around these arcs in $T^1 \Sigma$ in \Cref{fig:freepiecesectionbasic} bottom. The pictures labeled $x$ and $y$ are the arcs lying above the points $x$ and $y$ respectively in \Cref{fig:freepiecesectionbasic} top.

\begin{figure}
    \centering
    \fontsize{10pt}{10pt}
    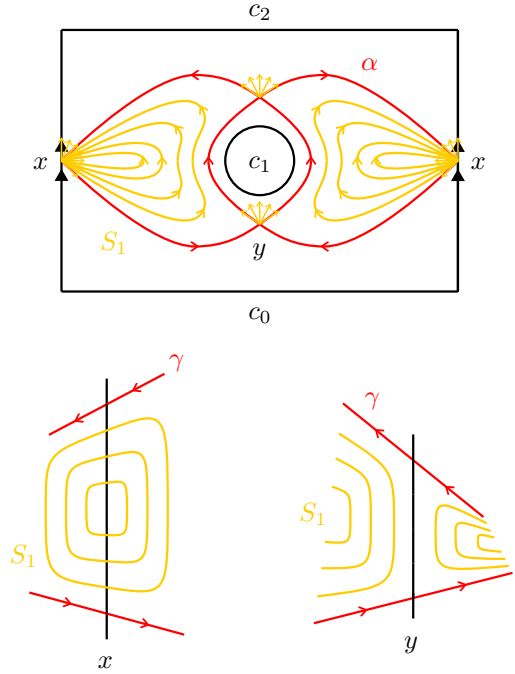
    \caption{Constructing one component $S_1$ of a basic section.}
    \label{fig:freepiecesectionbasic}
\end{figure}

The union of the two hexagons and the three arcs where they coincide gives us the partial Birkhoff section $S_1$.
Here, transversality of the interior of $S_1$ to the flow follows from the convexity of the loops.
It is straightforward to check that $S_1$ is a once-punctured torus. 
From \Cref{eq:indPH}, one then computes that the index of $\gamma$ as a boundary component of $S_1$ to be $-1$.

This is essentially the same construction as in \cite[Sections 4a, 4b]{Deh26}, where these partial Birkhoff sections are called `butterfly surfaces'.

\begin{claim} \label{claim:missedflowlines}
If $\ell$ is a flow line that misses $S_1$, then either $\ell$ lies over a boundary component $c_i$, or
    \begin{itemize}
        \item $\ell$ enters from, or limits in backward time on, a boundary torus lying above $c_i$, and
        \item $\ell$ exits at, or limits in forward time on, a boundary torus lying above $c_j$.
    \end{itemize}
    for some $i > j$.
\end{claim}
\begin{proof}
Let $\ell$ be a flow line that misses $S$.
Let $\beta$ be the geodesic line on $\Sigma$ that $\ell$ projects down to.

Observe that $\beta$ cannot cross $\alpha$ from $S_0$ into a triangle, for otherwise it must be positively tangent to a loop in \Cref{fig:freepiecesectionbasic} top at some point before exiting the triangle, that is, it must intersect $S$.
Similarly, $\beta$ cannot cross $\alpha$ through a self-intersection point from $S_0$ to $S_1$.
This reasoning shows that once $\beta$ enters $S_0$, it cannot leave again, thus it must exit through or limit onto $c_0$.
The symmetric fact is true in backward time with $S_2$ replaced by $S_0$.

Likewise, $\beta$ cannot cross from $S_1$ to $S_2$, nor from $S_1$ to $x$ through a triangle.
Thus once $\beta$ leaves $S_1$, it can only exit through or limit onto $c_0$.
Again, the symmetric fact is true in backward time with $c_2$ replaced by $c_0$.

The conclusion is that either $\beta$ stays within a complementary region $S_i$, in which case it must lie over $c_i$, or $\beta$ must enter through or limit onto $c_i$ in backward time, and exit through or limit onto $c_j$ in forward time, for some $i > j$.
\end{proof}

To construct $S_2$ and $S_3$, we cyclically permute the roles of $c_0, c_1, c_2$.
It is straightforward to check that $S_1, S_2, S_3$ have disjoint interiors away from the fibers over the self-intersection points of $\alpha$, from their descriptions as tangent vectors to a family of loops, see \Cref{fig:freepiecesectiondisjoint} left.
In \Cref{fig:freepiecesectiondisjoint} right, we illustrate the positioning of $S_1, S_2, S_3$ near a fiber lying above a self-intersection point $z$.
We see that two of the surfaces overlap along an interval on the fiber, but this interval of intersection can be removed by an small isotopy along flow lines. 
Thus up to this isotopy, we can assume that $S_1, S_2, S_3$ have disjoint interiors everywhere.

Once this is arranged, the union $S = S_1 \cup S_2 \cup S_3$ is a partial Birkhoff section consisting of three once-punctured tori.

\begin{figure}
    \centering
    \fontsize{10pt}{10pt}
    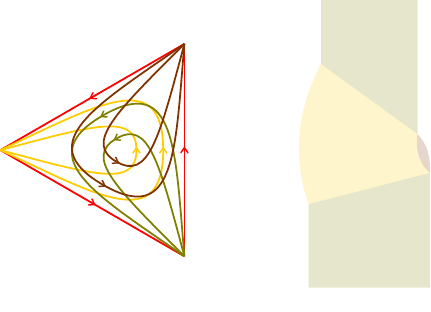
    \caption{The three partial Birkhoff sections $S_1, S_2, S_3$ have disjoint interiors away from the self-intersection points (left) and near the fiber lying above an intersection point $z$ (right).}
    \label{fig:freepiecesectiondisjoint}
\end{figure}

The corresponding versions of \Cref{claim:missedflowlines} holds for $S_2$ and $S_3$, thus $S$ intersects every flow line except for the closed orbits lying over the boundary components of $\Sigma$. 

\begin{lem} \label{lem:basicblock}
Let $\Sigma$ be a hyperbolic orbifold homeomorphic to
\begin{itemize}
    \item a pair of pants,
    \item an annulus with one cone point,
    \item a disc with two cone points,
    \item a sphere with three cone points.
\end{itemize}
Let $\phi$ be a $k$-fold fiberwise cover of the geodesic flow of $\Sigma$. 
Then there exists a partial Birkhoff section $S$ satisfying:
\begin{enumerate}
    \item every component $S$ has genus one,
    \item every boundary component of $S$ has index $-1$, and
    \item $S$ intersects every flow line except for the closed orbits lying over the boundary components of $\Sigma$. 
\end{enumerate}
\end{lem}
\begin{proof}
We have already proved the case when $\Sigma$ is a pair of pants and $k=1$, where $S$ is in fact a once-punctured torus.
Actually, the same construction works for the other orbifolds in the list --- one simply replaces the boundary components by cone points.

Some words might be warranted in the case when there is a cone point of order 2.
First of all, since we are assuming that $\Sigma$ is hyperbolic, there can be at most one such cone point.
Without loss of generality, suppose we replace the boundary component $c_1$ by this cone point $c$.
In this case, the oriented geodesic $\alpha$ degenerates into a geodesic arc connecting $c$ to itself, traversed back-and-forth.
There are now just two complementary regions $\Sigma_0$ and $\Sigma_2$.
The partial Birkhoff section $S_1$ degenerates into the collection of unit tangent vectors that point from $\Sigma_0$ to $\Sigma_2$ along $\alpha$.
See \Cref{fig:freepiecesectioncone}, where the picture on the right is drawn for the cover where $c$ is lifted to a regular point $\widetilde{c}$.
The fiber over $c$ intersects $S_1$ in two arcs, cutting it into a octagon. 
From this, one can check that $S_1$ remains a once-puncture torus.

\begin{figure}
    \centering
    \fontsize{10pt}{10pt}
    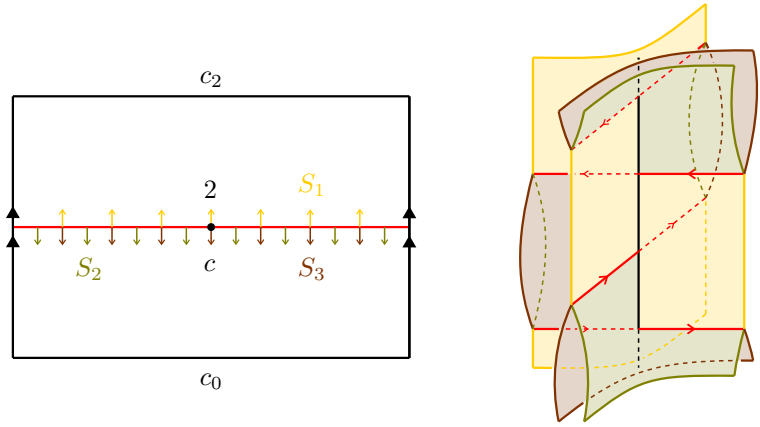
    \caption{The partial Birkhoff sections $S_1, S_2, S_3$ when the boundary component $c_1$ is replaced with a cone point $c$ of order 2. The picture on the right is drawn for the cover where $c$ is lifted to a regular point $\widetilde{c}$.}
    \label{fig:freepiecesectioncone}
\end{figure}

Similarly, both $S_2$ and $S_3$ degenerate into the collection of unit tangent vectors that point from $\Sigma_2$ to $\Sigma_0$ along $\alpha$.
See \Cref{fig:freepiecesectioncone} again. 

For $k > 1$, we can simply take the preimage of the partial Birkhoff section constructed for the $k=1$ case. Since the cover is cyclic, each once-punctured torus lifts to a union of (possibly multiply) punctured tori, with each boundary component lifting homeomorphically.
Item (2) follows.
Item (3) is also preserved under fiberwise covers.
\end{proof}

In the sequel, we will refer to the list of orbifolds in the statement of \Cref{lem:basicblock} as the \textbf{basic orbifolds}, and refer to a partial Birkhoff section as in \Cref{lem:basicblock} as a \textbf{basic section}.

\subsection{Assembling basic sections} \label{subsec:assembleblocks}

In this subsection, we will prove \Cref{prop:freepiecepartialsection} by assembling the basic sections of \Cref{lem:basicblock}.

\begin{lem} \label{lem:cutintobasicorbifolds}
Every hyperbolic orbifold $\Sigma$ can be cut along curves and arcs between cone points of order 2 into basic orbifolds.
\end{lem}
\begin{proof}
Let $R$ be the surface with boundary obtained from $\Sigma$ by removing a neighborhood of each cone point.
Then since $\chi(R) \leq \chi^\orb(\Sigma) < 0$, $R$ admits a pants decomposition.
Filling back in the cone points, each pants gives a basic orbifold, unless it is non-hyperbolic. But since $\Sigma$ itself is hyperbolic, this can only happen if two boundary components are filled to cone points of order 2, in which case we have a disc with two cone points of order 2.

We cannot have two such discs be adjacent to each other, for otherwise $\Sigma$ is a sphere with four cone points of order 2, which is not hyperbolic. 
Thus we can homotope the boundary circle of each such disc into the arc between the two cone points.
The effect of this on the complementary surfaces is to remove the disc, leaving us only with the basic orbifolds.
\end{proof}

\begin{proof}[Proof of \Cref{prop:freepiecepartialsection}]
Apply \Cref{lem:cutintobasicorbifolds} and decompose $\Sigma$ into basic orbifolds. 
Let $S$ be the union of basic sections over all basic orbifolds.
By \Cref{lem:basicblock}(3), $S$ intersects every flow line except for the closed orbits lying over the decomposition curves and arcs.
\end{proof}

\section{Proof of main theorem}

With \Cref{prop:surtototallyperiodicrecog} and \Cref{prop:freepiecepartialsection} in hand, we are ready to prove \Cref{thm:graphmanifoldtototallyperiodic}. 

\begin{proof}[Proof of \Cref{thm:graphmanifoldtototallyperiodic}]
Let $\phi$ be a pseudo-Anosov flow on a graph manifold.
We apply \Cref{prop:toribetweenperiodicpieces} to get a transverse torus between every pair of adjacent periodic pieces.
Next, by \Cref{prop:freepiece}, each free piece is orbit semi-equivalent to the $k$-fold fiberwise cover of the geodesic flow of a hyperbolic orbifold with geodesic boundary. We pull back the partial Birkhoff section from \Cref{prop:freepiecepartialsection} to get a partial Birkhoff section with the same properties as listed in the proposition.

Let $S$ be the union of the transverse tori and the partial Birkhoff sections in the free pieces.
Every component of $S$ has genus one, and every boundary component of $S$ has index $-1$.

We turn to analyze the flow lines that miss $S$.
Every flow line that passes from one periodic piece to another intersects one of the transverse tori in $S$.
Meanwhile, by \Cref{prop:freepiecepartialsection}(3), every flow line that passes through the interior of a free piece intersects the component of $S$ in that free piece.
Thus any flow line that misses $S$ must be contained in a periodic piece or a JSJ torus.
But by \Cref{prop:periodicpiece}, the set of flow lines contained in a periodic piece is a finite collection of closed orbits. 
We conclude that the set of flow lines that are disjoint from $S$ is a finite collection of closed orbits as well.

Thus we are done by \Cref{prop:surtototallyperiodicrecog}.
\end{proof}

\bibliographystyle{alphaurl}

\bibliography{bib.bib}

\end{document}

%% file: paflow.pdf_tex
\begingroup%
  \makeatletter%
  \providecommand\color[2][]{%
    \errmessage{(Inkscape) Color is used for the text in Inkscape, but the package 'color.sty' is not loaded}%
    \renewcommand\color[2][]{}%
  }%
  \providecommand\transparent[1]{%
    \errmessage{(Inkscape) Transparency is used (non-zero) for the text in Inkscape, but the package 'transparent.sty' is not loaded}%
    \renewcommand\transparent[1]{}%
  }%
  \providecommand\rotatebox[2]{#2}%
  \newcommand*\fsize{\dimexpr\f@size pt\relax}%
  \newcommand*\lineheight[1]{\fontsize{\fsize}{#1\fsize}\selectfont}%
  \ifx\svgwidth\undefined%
    \setlength{\unitlength}{212.78949647bp}%
    \ifx\svgscale\undefined%
      \relax%
    \else%
      \setlength{\unitlength}{\unitlength * \real{\svgscale}}%
    \fi%
  \else%
    \setlength{\unitlength}{\svgwidth}%
  \fi%
  \global\let\svgwidth\undefined%
  \global\let\svgscale\undefined%
  \makeatother%
  \begin{picture}(1,0.48600821)%
    \lineheight{1}%
    \setlength\tabcolsep{0pt}%
    \put(0,0){\includegraphics[width=\unitlength,page=1]{paflow.pdf}}%
  \end{picture}%
\endgroup%

%% file: gfsurgery.pdf_tex
\begingroup%
  \makeatletter%
  \providecommand\color[2][]{%
    \errmessage{(Inkscape) Color is used for the text in Inkscape, but the package 'color.sty' is not loaded}%
    \renewcommand\color[2][]{}%
  }%
  \providecommand\transparent[1]{%
    \errmessage{(Inkscape) Transparency is used (non-zero) for the text in Inkscape, but the package 'transparent.sty' is not loaded}%
    \renewcommand\transparent[1]{}%
  }%
  \providecommand\rotatebox[2]{#2}%
  \newcommand*\fsize{\dimexpr\f@size pt\relax}%
  \newcommand*\lineheight[1]{\fontsize{\fsize}{#1\fsize}\selectfont}%
  \ifx\svgwidth\undefined%
    \setlength{\unitlength}{365.25198388bp}%
    \ifx\svgscale\undefined%
      \relax%
    \else%
      \setlength{\unitlength}{\unitlength * \real{\svgscale}}%
    \fi%
  \else%
    \setlength{\unitlength}{\svgwidth}%
  \fi%
  \global\let\svgwidth\undefined%
  \global\let\svgscale\undefined%
  \makeatother%
  \begin{picture}(1,0.2825237)%
    \lineheight{1}%
    \setlength\tabcolsep{0pt}%
    \put(0,0){\includegraphics[width=\unitlength,page=1]{gfsurgery.pdf}}%
    \put(0.2611667,0.1804316){\color[rgb]{0,0,0}\makebox(0,0)[lt]{\lineheight{1.25}\smash{\begin{tabular}[t]{l}blow up\end{tabular}}}}%
    \put(0.25360017,0.12720296){\color[rgb]{0,0,0}\makebox(0,0)[lt]{\lineheight{1.25}\smash{\begin{tabular}[t]{l}blow down\end{tabular}}}}%
    \put(0.63898877,0.1804319){\color[rgb]{0,0,0}\makebox(0,0)[lt]{\lineheight{1.25}\smash{\begin{tabular}[t]{l}blow down\end{tabular}}}}%
    \put(0.65195568,0.12720326){\color[rgb]{0,0,0}\makebox(0,0)[lt]{\lineheight{1.25}\smash{\begin{tabular}[t]{l}blow up\end{tabular}}}}%
  \end{picture}%
\endgroup%

%% file: maximallytransverse.pdf_tex
\begingroup%
  \makeatletter%
  \providecommand\color[2][]{%
    \errmessage{(Inkscape) Color is used for the text in Inkscape, but the package 'color.sty' is not loaded}%
    \renewcommand\color[2][]{}%
  }%
  \providecommand\transparent[1]{%
    \errmessage{(Inkscape) Transparency is used (non-zero) for the text in Inkscape, but the package 'transparent.sty' is not loaded}%
    \renewcommand\transparent[1]{}%
  }%
  \providecommand\rotatebox[2]{#2}%
  \newcommand*\fsize{\dimexpr\f@size pt\relax}%
  \newcommand*\lineheight[1]{\fontsize{\fsize}{#1\fsize}\selectfont}%
  \ifx\svgwidth\undefined%
    \setlength{\unitlength}{438.81059782bp}%
    \ifx\svgscale\undefined%
      \relax%
    \else%
      \setlength{\unitlength}{\unitlength * \real{\svgscale}}%
    \fi%
  \else%
    \setlength{\unitlength}{\svgwidth}%
  \fi%
  \global\let\svgwidth\undefined%
  \global\let\svgscale\undefined%
  \makeatother%
  \begin{picture}(1,0.28563759)%
    \lineheight{1}%
    \setlength\tabcolsep{0pt}%
    \put(0,0){\includegraphics[width=\unitlength,page=1]{maximallytransverse.pdf}}%
  \end{picture}%
\endgroup%

%% file: freepiece.pdf_tex
\begingroup%
  \makeatletter%
  \providecommand\color[2][]{%
    \errmessage{(Inkscape) Color is used for the text in Inkscape, but the package 'color.sty' is not loaded}%
    \renewcommand\color[2][]{}%
  }%
  \providecommand\transparent[1]{%
    \errmessage{(Inkscape) Transparency is used (non-zero) for the text in Inkscape, but the package 'transparent.sty' is not loaded}%
    \renewcommand\transparent[1]{}%
  }%
  \providecommand\rotatebox[2]{#2}%
  \newcommand*\fsize{\dimexpr\f@size pt\relax}%
  \newcommand*\lineheight[1]{\fontsize{\fsize}{#1\fsize}\selectfont}%
  \ifx\svgwidth\undefined%
    \setlength{\unitlength}{347.73408905bp}%
    \ifx\svgscale\undefined%
      \relax%
    \else%
      \setlength{\unitlength}{\unitlength * \real{\svgscale}}%
    \fi%
  \else%
    \setlength{\unitlength}{\svgwidth}%
  \fi%
  \global\let\svgwidth\undefined%
  \global\let\svgscale\undefined%
  \makeatother%
  \begin{picture}(1,0.63553925)%
    \lineheight{1}%
    \setlength\tabcolsep{0pt}%
    \put(0,0){\includegraphics[width=\unitlength,page=1]{freepiece.pdf}}%
    \put(0.31529608,0.00475172){\color[rgb]{0,0,0}\makebox(0,0)[lt]{\lineheight{1.25}\smash{\begin{tabular}[t]{l}$\Sigma$\end{tabular}}}}%
    \put(0.29871542,0.30119658){\color[rgb]{0,0,0}\makebox(0,0)[lt]{\lineheight{1.25}\smash{\begin{tabular}[t]{l}$T^1 \Sigma$\end{tabular}}}}%
    \put(0.85019055,0.00475172){\color[rgb]{0,0,0}\makebox(0,0)[lt]{\lineheight{1.25}\smash{\begin{tabular}[t]{l}$\Sigma$\end{tabular}}}}%
    \put(0.8336099,0.30119658){\color[rgb]{0,0,0}\makebox(0,0)[lt]{\lineheight{1.25}\smash{\begin{tabular}[t]{l}$T^1 \Sigma$\end{tabular}}}}%
    \put(0.411845,0.32250726){\color[rgb]{0,0,0}\makebox(0,0)[lt]{\lineheight{1.25}\smash{\begin{tabular}[t]{l}hyperbolic\end{tabular}}}}%
    \put(0.411845,0.28636522){\color[rgb]{0,0,0}\makebox(0,0)[lt]{\lineheight{1.25}\smash{\begin{tabular}[t]{l}blow up\end{tabular}}}}%
  \end{picture}%
\endgroup%

%% file: roundhandle.pdf_tex
\begingroup%
  \makeatletter%
  \providecommand\color[2][]{%
    \errmessage{(Inkscape) Color is used for the text in Inkscape, but the package 'color.sty' is not loaded}%
    \renewcommand\color[2][]{}%
  }%
  \providecommand\transparent[1]{%
    \errmessage{(Inkscape) Transparency is used (non-zero) for the text in Inkscape, but the package 'transparent.sty' is not loaded}%
    \renewcommand\transparent[1]{}%
  }%
  \providecommand\rotatebox[2]{#2}%
  \newcommand*\fsize{\dimexpr\f@size pt\relax}%
  \newcommand*\lineheight[1]{\fontsize{\fsize}{#1\fsize}\selectfont}%
  \ifx\svgwidth\undefined%
    \setlength{\unitlength}{100.1096386bp}%
    \ifx\svgscale\undefined%
      \relax%
    \else%
      \setlength{\unitlength}{\unitlength * \real{\svgscale}}%
    \fi%
  \else%
    \setlength{\unitlength}{\svgwidth}%
  \fi%
  \global\let\svgwidth\undefined%
  \global\let\svgscale\undefined%
  \makeatother%
  \begin{picture}(1,1.07382299)%
    \lineheight{1}%
    \setlength\tabcolsep{0pt}%
    \put(0,0){\includegraphics[width=\unitlength,page=1]{roundhandle.pdf}}%
  \end{picture}%
\endgroup%

%% file: bookofIbundlesrecog.pdf_tex
\begingroup%
  \makeatletter%
  \providecommand\color[2][]{%
    \errmessage{(Inkscape) Color is used for the text in Inkscape, but the package 'color.sty' is not loaded}%
    \renewcommand\color[2][]{}%
  }%
  \providecommand\transparent[1]{%
    \errmessage{(Inkscape) Transparency is used (non-zero) for the text in Inkscape, but the package 'transparent.sty' is not loaded}%
    \renewcommand\transparent[1]{}%
  }%
  \providecommand\rotatebox[2]{#2}%
  \newcommand*\fsize{\dimexpr\f@size pt\relax}%
  \newcommand*\lineheight[1]{\fontsize{\fsize}{#1\fsize}\selectfont}%
  \ifx\svgwidth\undefined%
    \setlength{\unitlength}{124.51096542bp}%
    \ifx\svgscale\undefined%
      \relax%
    \else%
      \setlength{\unitlength}{\unitlength * \real{\svgscale}}%
    \fi%
  \else%
    \setlength{\unitlength}{\svgwidth}%
  \fi%
  \global\let\svgwidth\undefined%
  \global\let\svgscale\undefined%
  \makeatother%
  \begin{picture}(1,0.89883102)%
    \lineheight{1}%
    \setlength\tabcolsep{0pt}%
    \put(0,0){\includegraphics[width=\unitlength,page=1]{bookofIbundlesrecog.pdf}}%
    \put(0.06751217,0.01132569){\color[rgb]{0,0,0}\makebox(0,0)[lt]{\lineheight{1.25}\smash{\begin{tabular}[t]{l}$S$\end{tabular}}}}%
    \put(0.87808522,0.242366){\color[rgb]{0,0,0}\makebox(0,0)[lt]{\lineheight{1.25}\smash{\begin{tabular}[t]{l}$R$\end{tabular}}}}%
  \end{picture}%
\endgroup%

%% file: freepiecesectioncurve.pdf_tex
\begingroup%
  \makeatletter%
  \providecommand\color[2][]{%
    \errmessage{(Inkscape) Color is used for the text in Inkscape, but the package 'color.sty' is not loaded}%
    \renewcommand\color[2][]{}%
  }%
  \providecommand\transparent[1]{%
    \errmessage{(Inkscape) Transparency is used (non-zero) for the text in Inkscape, but the package 'transparent.sty' is not loaded}%
    \renewcommand\transparent[1]{}%
  }%
  \providecommand\rotatebox[2]{#2}%
  \newcommand*\fsize{\dimexpr\f@size pt\relax}%
  \newcommand*\lineheight[1]{\fontsize{\fsize}{#1\fsize}\selectfont}%
  \ifx\svgwidth\undefined%
    \setlength{\unitlength}{93.52534617bp}%
    \ifx\svgscale\undefined%
      \relax%
    \else%
      \setlength{\unitlength}{\unitlength * \real{\svgscale}}%
    \fi%
  \else%
    \setlength{\unitlength}{\svgwidth}%
  \fi%
  \global\let\svgwidth\undefined%
  \global\let\svgscale\undefined%
  \makeatother%
  \begin{picture}(1,0.8974646)%
    \lineheight{1}%
    \setlength\tabcolsep{0pt}%
    \put(0,0){\includegraphics[width=\unitlength,page=1]{freepiecesectioncurve.pdf}}%
    \put(0.85925695,0.64667124){\color[rgb]{1,0,0}\makebox(0,0)[lt]{\lineheight{1.25}\smash{\begin{tabular}[t]{l}$\alpha$\end{tabular}}}}%
  \end{picture}%
\endgroup%

%% file: freepiecesectionbasic.pdf_tex
\begingroup%
  \makeatletter%
  \providecommand\color[2][]{%
    \errmessage{(Inkscape) Color is used for the text in Inkscape, but the package 'color.sty' is not loaded}%
    \renewcommand\color[2][]{}%
  }%
  \providecommand\transparent[1]{%
    \errmessage{(Inkscape) Transparency is used (non-zero) for the text in Inkscape, but the package 'transparent.sty' is not loaded}%
    \renewcommand\transparent[1]{}%
  }%
  \providecommand\rotatebox[2]{#2}%
  \newcommand*\fsize{\dimexpr\f@size pt\relax}%
  \newcommand*\lineheight[1]{\fontsize{\fsize}{#1\fsize}\selectfont}%
  \ifx\svgwidth\undefined%
    \setlength{\unitlength}{191.067601bp}%
    \ifx\svgscale\undefined%
      \relax%
    \else%
      \setlength{\unitlength}{\unitlength * \real{\svgscale}}%
    \fi%
  \else%
    \setlength{\unitlength}{\svgwidth}%
  \fi%
  \global\let\svgwidth\undefined%
  \global\let\svgscale\undefined%
  \makeatother%
  \begin{picture}(1,1.33608117)%
    \lineheight{1}%
    \setlength\tabcolsep{0pt}%
    \put(0,0){\includegraphics[width=\unitlength,page=1]{freepiecesectionbasic.pdf}}%
    \put(0.46734152,1.29580578){\color[rgb]{0,0,0}\makebox(0,0)[lt]{\lineheight{1.25}\smash{\begin{tabular}[t]{l}$c_2$\end{tabular}}}}%
    \put(0.46734152,0.6991578){\color[rgb]{0,0,0}\makebox(0,0)[lt]{\lineheight{1.25}\smash{\begin{tabular}[t]{l}$c_0$\end{tabular}}}}%
    \put(0.46734152,0.99748182){\color[rgb]{0,0,0}\makebox(0,0)[lt]{\lineheight{1.25}\smash{\begin{tabular}[t]{l}$c_1$\end{tabular}}}}%
    \put(0.90610967,0.99488924){\color[rgb]{0,0,0}\makebox(0,0)[lt]{\lineheight{1.25}\smash{\begin{tabular}[t]{l}$x$\end{tabular}}}}%
    \put(0.17017395,0.00922545){\color[rgb]{0,0,0}\makebox(0,0)[lt]{\lineheight{1.25}\smash{\begin{tabular}[t]{l}$x$\end{tabular}}}}%
    \put(0.04254104,0.99488924){\color[rgb]{0,0,0}\makebox(0,0)[lt]{\lineheight{1.25}\smash{\begin{tabular}[t]{l}$x$\end{tabular}}}}%
    \put(0.47519215,0.82476791){\color[rgb]{0,0,0}\makebox(0,0)[lt]{\lineheight{1.25}\smash{\begin{tabular}[t]{l}$y$\end{tabular}}}}%
    \put(0.77467239,0.04847805){\color[rgb]{0,0,0}\makebox(0,0)[lt]{\lineheight{1.25}\smash{\begin{tabular}[t]{l}$y$\end{tabular}}}}%
    \put(0.68993471,1.19030573){\color[rgb]{1,0,0}\makebox(0,0)[lt]{\lineheight{1.25}\smash{\begin{tabular}[t]{l}$\alpha$\end{tabular}}}}%
    \put(0.31130445,0.60616541){\color[rgb]{1,0,0}\makebox(0,0)[lt]{\lineheight{1.25}\smash{\begin{tabular}[t]{l}$\gamma$\end{tabular}}}}%
    \put(0.69699506,0.53141637){\color[rgb]{1,0,0}\makebox(0,0)[lt]{\lineheight{1.25}\smash{\begin{tabular}[t]{l}$\gamma$\end{tabular}}}}%
    \put(0.17463776,0.83458463){\color[rgb]{1,0.8,0}\makebox(0,0)[lt]{\lineheight{1.25}\smash{\begin{tabular}[t]{l}$S_1$\end{tabular}}}}%
    \put(-0.0039611,0.21616843){\color[rgb]{1,0.8,0}\makebox(0,0)[lt]{\lineheight{1.25}\smash{\begin{tabular}[t]{l}$S_1$\end{tabular}}}}%
    \put(0.56747791,0.300319){\color[rgb]{1,0.8,0}\makebox(0,0)[lt]{\lineheight{1.25}\smash{\begin{tabular}[t]{l}$S_1$\end{tabular}}}}%
  \end{picture}%
\endgroup%

%% file: freepiecesectiondisjoint.pdf_tex
\begingroup%
  \makeatletter%
  \providecommand\color[2][]{%
    \errmessage{(Inkscape) Color is used for the text in Inkscape, but the package 'color.sty' is not loaded}%
    \renewcommand\color[2][]{}%
  }%
  \providecommand\transparent[1]{%
    \errmessage{(Inkscape) Transparency is used (non-zero) for the text in Inkscape, but the package 'transparent.sty' is not loaded}%
    \renewcommand\transparent[1]{}%
  }%
  \providecommand\rotatebox[2]{#2}%
  \newcommand*\fsize{\dimexpr\f@size pt\relax}%
  \newcommand*\lineheight[1]{\fontsize{\fsize}{#1\fsize}\selectfont}%
  \ifx\svgwidth\undefined%
    \setlength{\unitlength}{206.84184878bp}%
    \ifx\svgscale\undefined%
      \relax%
    \else%
      \setlength{\unitlength}{\unitlength * \real{\svgscale}}%
    \fi%
  \else%
    \setlength{\unitlength}{\svgwidth}%
  \fi%
  \global\let\svgwidth\undefined%
  \global\let\svgscale\undefined%
  \makeatother%
  \begin{picture}(1,0.72359576)%
    \lineheight{1}%
    \setlength\tabcolsep{0pt}%
    \put(0,0){\includegraphics[width=\unitlength,page=1]{freepiecesectiondisjoint.pdf}}%
    \put(0.10397203,0.36170183){\color[rgb]{1,0.8,0}\makebox(0,0)[lt]{\lineheight{1.25}\smash{\begin{tabular}[t]{l}$S_1$\end{tabular}}}}%
    \put(0.33778404,0.22794147){\color[rgb]{0.50196078,0.50196078,0}\makebox(0,0)[lt]{\lineheight{1.25}\smash{\begin{tabular}[t]{l}$S_2$\end{tabular}}}}%
    \put(0.33777902,0.49564318){\color[rgb]{0.50196078,0.2,0}\makebox(0,0)[lt]{\lineheight{1.25}\smash{\begin{tabular}[t]{l}$S_3$\end{tabular}}}}%
    \put(0.45125382,0.36413768){\color[rgb]{1,0,0}\makebox(0,0)[lt]{\lineheight{1.25}\smash{\begin{tabular}[t]{l}$\alpha$\end{tabular}}}}%
    \put(0,0){\includegraphics[width=\unitlength,page=2]{freepiecesectiondisjoint.pdf}}%
    \put(0.44059639,0.09260801){\color[rgb]{0,0,0}\makebox(0,0)[lt]{\lineheight{1.25}\smash{\begin{tabular}[t]{l}$z$\end{tabular}}}}%
    \put(0.84137387,0.00852191){\color[rgb]{0,0,0}\makebox(0,0)[lt]{\lineheight{1.25}\smash{\begin{tabular}[t]{l}$z$\end{tabular}}}}%
    \put(0,0){\includegraphics[width=\unitlength,page=3]{freepiecesectiondisjoint.pdf}}%
  \end{picture}%
\endgroup%

%% file: freepiecesectioncone.pdf_tex
\begingroup%
  \makeatletter%
  \providecommand\color[2][]{%
    \errmessage{(Inkscape) Color is used for the text in Inkscape, but the package 'color.sty' is not loaded}%
    \renewcommand\color[2][]{}%
  }%
  \providecommand\transparent[1]{%
    \errmessage{(Inkscape) Transparency is used (non-zero) for the text in Inkscape, but the package 'transparent.sty' is not loaded}%
    \renewcommand\transparent[1]{}%
  }%
  \providecommand\rotatebox[2]{#2}%
  \newcommand*\fsize{\dimexpr\f@size pt\relax}%
  \newcommand*\lineheight[1]{\fontsize{\fsize}{#1\fsize}\selectfont}%
  \ifx\svgwidth\undefined%
    \setlength{\unitlength}{283.43037523bp}%
    \ifx\svgscale\undefined%
      \relax%
    \else%
      \setlength{\unitlength}{\unitlength * \real{\svgscale}}%
    \fi%
  \else%
    \setlength{\unitlength}{\svgwidth}%
  \fi%
  \global\let\svgwidth\undefined%
  \global\let\svgscale\undefined%
  \makeatother%
  \begin{picture}(1,0.5579783)%
    \lineheight{1}%
    \setlength\tabcolsep{0pt}%
    \put(0,0){\includegraphics[width=\unitlength,page=1]{freepiecesectioncone.pdf}}%
    \put(0.2577996,0.45182291){\color[rgb]{0,0,0}\makebox(0,0)[lt]{\lineheight{1.25}\smash{\begin{tabular}[t]{l}$c_2$\end{tabular}}}}%
    \put(0.2577996,0.04960742){\color[rgb]{0,0,0}\makebox(0,0)[lt]{\lineheight{1.25}\smash{\begin{tabular}[t]{l}$c_0$\end{tabular}}}}%
    \put(0.26309191,0.20108374){\color[rgb]{0,0,0}\makebox(0,0)[lt]{\lineheight{1.25}\smash{\begin{tabular}[t]{l}$c$\end{tabular}}}}%
    \put(0.26309191,0.29834599){\color[rgb]{0,0,0}\makebox(0,0)[lt]{\lineheight{1.25}\smash{\begin{tabular}[t]{l}$2$\end{tabular}}}}%
    \put(0.38827158,0.30607516){\color[rgb]{1,0.8,0}\makebox(0,0)[lt]{\lineheight{1.25}\smash{\begin{tabular}[t]{l}$S_1$\end{tabular}}}}%
    \put(0.38827158,0.19493668){\color[rgb]{0.50196078,0.2,0}\makebox(0,0)[lt]{\lineheight{1.25}\smash{\begin{tabular}[t]{l}$S_3$\end{tabular}}}}%
    \put(0.09190221,0.19493668){\color[rgb]{0.50196078,0.50196078,0}\makebox(0,0)[lt]{\lineheight{1.25}\smash{\begin{tabular}[t]{l}$S_2$\end{tabular}}}}%
    \put(0,0){\includegraphics[width=\unitlength,page=2]{freepiecesectioncone.pdf}}%
  \end{picture}%
\endgroup%